\documentclass[11pt]{amsart}

\usepackage{amsfonts}
\usepackage{amssymb}
\usepackage{mathrsfs}
\usepackage{color}

\newtheorem{thm}{Theorem}[section]
\newtheorem{cor}[thm]{Corollary}
\newtheorem{lem}[thm]{Lemma}

\theoremstyle{definition}
\newtheorem{defn}[thm]{Definition}
\theoremstyle{remark}
\newtheorem{rem}[thm]{Remark}
\numberwithin{equation}{section}
\newtheorem{exmp}{Example}[section]

\begin{document}
\title[Enriched Contraction mappings in $2$-normed spaces]{FIXED POINTS OF
ENRICHED CONTRACTION MAPPINGS IN $2$-NORMED SPACES}
\author{Rizwan Anjum}
\address{Department of Mathematics, Riphah Institute of Computing and Applied Sciences, Riphah International University, Lahore, Pakistan.}
\email{rizwan.anjum@riphah.edu.pk}
\author{Mujahid Abbas}
\address{Department of Mathematics, Government College University Katchery
Road, Lahore, Pakistan and Department of Mathematics and Applied
Mathematics, University of Pretoria Hatfield 002, Pretoria, South Africa}
\email{abbas.mujahid@gmail.com }
\thanks{}
\thanks{}
\date{}
\dedicatory{}

\begin{abstract}
Very recently, Berinde and P\u{a}curar obtained in [V. Berinde and M.  P\u{a}curar, Approximating fixed points of
enriched contractions in Banach spaces. Journal of Fixed Point Theory and
Applications.  \textbf{22}(2) (2020), 1--10.] an interesting
generalization of the Banach contractive mapping principle  in
the framework of Banach spaces. The aim of this paper is to prove some fixed
point theorems which extend the results of Berinde and P\u{a}curar  to the case of $2$%
-normed spaces.

\vspace{0,1cm} \noindent 2020 \textit{Mathematics Subject Classification.}
46B20. 47H09. 47H10.

\vspace{0,1cm} \noindent \textit{Keywords: Fixed point, enriched contraction, $2$%
-normed space }

\end{abstract}

\maketitle

\commby{}

\section{Introduction and Preliminaries}

Let $(X,d)$ be a metric space and let $T:X\rightarrow X$ be a self-mapping.
We denote the set $\{x\in X:Tx=x\}$ of \emph{fixed points} of $T$ by
$Fix(T)$. Define $T^{0}=I$ (the identity map on $X$) and $%
T^{n}=T^{n-1}\circ T$, called the $n^{th}$ \emph{iterate of} $T$ for $n\geq
1 $. Let $x_{0}\in X$ be an initial point in order to approximate the
solution of a functional equation $Tx=x$ in $X.$ A sequence $%
\{x_{n}\}_{n=0}^{\infty } $ in $X$ is called a \emph{Picard iteration
associated with} $T$ if%
\begin{equation*}
x_{n}=T^{n}x_{0},\ \ n=1,2,\dotsc
\end{equation*}%
\newline
A mapping $T:X\rightarrow X$ is called a \emph{Banach contraction} if there
exists $k\in \lbrack 0,1)$ such that, for all $x,y\in X$, we have
\begin{equation}
d(Tx,Ty)\leq k\,d(x,y).  \label{ban}
\end{equation}%
If there exists $x^{\ast }$ in $X$ such that $Fix(T)=\{x^{\ast }\}$
and the Picard iteration associated with $T$ converges to $x^{\ast }$ for
any initial point $x_{0}$ in $X$, then $T$ is called a \emph{Picard operator}
(see \cite{G-book,Rus}).

Banach fixed point theorem reads as follows.

\begin{thm}
\label{mathematics}\cite{10} Let $(X,d)$ be a complete metric space and let $%
T:X\rightarrow X$ be a Banach contraction. Then $T$ is a Picard operator.
\end{thm}

In many cases (such as solving nonlinear functional equations, optimization
problems, variational inequalities problems, split feasibility problems and
equilibrium point problems), the transformation of the given problem into an
appropriate fixed point problem for certain operator requires an appropriate
space on which the corresponding operator acts and where the solution is
identified. Moreover, in an equivalent form, operators are required to
satisfy certain contractive conditions to guarantee the existence of the
solution of the corresponding problems and the ability for approximating
them. Banach's fixed point theorem is a classical example which addresses
these issues. This principle has been extended and generalized either by
improving the contractivity conditions on the mapping $T$ or by weakening
the distance structure of the underlying space $X$ (see \cite{Anjum1}-\cite{Anjum5}, \cite{10,11,Brzdek,Chouhan},
\cite{Hatikrishnan}-\cite{Rumlawang}, \cite{13}, \cite{sharma} and references therein).

Recently, Berinde and P\u{a}curar \cite{BP} introduced a new family of
mappings which includes the class of Banach contractive mappings as a
special class as follows.

Let $(X,\left\Vert \cdot \right\Vert )$ be a linear normed space. A mapping $%
T:X\rightarrow X$ is said to be a $(b,\theta )$\emph{-enriched contraction}
if there exist $b\in \lbrack 0,\infty )$ and $\theta \in \lbrack 0,b+1)$
such that
\begin{equation}
\left\Vert b(x-y)+Tx-Ty\right\Vert \leq \theta \left\Vert x-y\right\Vert
\qquad \text{for all }x,y\in X.  \label{enriched}
\end{equation}%
Note that a $(0,\theta )$-enriched contraction mapping is a Banach
contraction.

If $M$ is a convex subset of a linear space $X$, $T:M\rightarrow M$ is a
self-mapping on $M$ and $\lambda \in (0,1]$, then the mapping $T_{\lambda
}:M\rightarrow M$ given by%
\begin{equation}
T_{\lambda }(x)=(1-\lambda )x+\lambda Tx\qquad \text{for all }x\in M
\label{E7}
\end{equation}%
is called an \emph{averaged mapping}. A Picard iteration $%
\{x_{n}:n=0,1,2,...\}$ corresponding to an averaged mapping $T_{\lambda }$
is called a \emph{Krasnoselskij iteration}.

It was proved in \cite{BP} that any enriched contraction on a Banach space
has an unique fixed point, which can be approximated by means of the
Krasnoselskij iterative scheme.

The framework in which we are going to develop the main results of this paper
is the setting of $2$-normed spaces, which were firstly introduced by
G\"{a}hler in \cite{gour} as a way to measure the linear dependence of two vectors
in a linear space.

\begin{defn}
(Freese and Cho \cite{Freese}) Let $X$ be a real linear space of dimension
greater than $1$. A mapping $\left\Vert \cdot ,\cdot \right\Vert :X\times
X\rightarrow \mathbb{R}$ is said to be a $2$\emph{-norm} if for any $%
x,y,z\in X$ and $\alpha \in \mathbb{R}$, the following conditions hold:

\begin{enumerate}
\item $\left\Vert x,y\right\Vert \geq 0,$\newline
$\left\Vert x,y\right\Vert =0$ if and only if $x$ and $y$ are linearly
dependent vectors,

\item $\left\Vert x,y\right\Vert =\left\Vert y,x\right\Vert ,$

\item $\left\Vert \alpha x,y\right\Vert =\left\vert \alpha \right\vert
\left\Vert x,y\right\Vert ,$

\item $\left\Vert x+y,z\right\Vert \leq \left\Vert x,z\right\Vert
+\left\Vert y,z\right\Vert .$
\end{enumerate}

The pair $(X,\left\Vert \cdot ,\cdot \right\Vert )$ is called a $2$\emph{%
-normed space}.
\end{defn}

A classical example of a $2$-norm $\left\Vert \cdot ,\cdot \right\Vert $ on $%
\mathbb{R}^{2}$ is given by
\begin{equation*}
\left\Vert u,v\right\Vert =\left\vert u_{1}v_{2}-u_{2}v_{1}\right\vert ,
\end{equation*}%
where $u=(u_{1},u_{2}),v=(v_{1},v_{2})\in \mathbb{R}^{2}$. \newline
Consistently with \cite{Freese}, \cite{gour}, \cite{Hatikrishnan}, \cite%
{Malceski}, \cite{Topo} and \cite{white}, the following results and
definitions will be needed henceforth.

Every $2$-normed space is a locally convex topological vector space. In
fact, associated to a given $z\in X$, we can define a function $p_{z}(\cdot
) $ on $X$ by
\begin{equation*}
p_{z}(x)=\left\Vert x,z\right\Vert \qquad \text{for all }x\in X,
\end{equation*}%
which satisfies the following properties for all $x,y\in X$ and all $\alpha
\in \mathbb{R}$:
\begin{eqnarray*}
&&p_{z}(x+y)\leq p_{z}(x)+p_{z}(y)\text{\qquad and } \\
&&p_{z}(\alpha x)=\left\vert \alpha \right\vert p_{z}(x).
\end{eqnarray*}%
A function $p:X\rightarrow \mathbb{R}$ satisfying the above conditions is
called a \emph{seminorm} on $X$. Since $X$ is a linear space of dimension
greater than or equal to 2, corresponding to each $x\neq 0$, there exist
some $z\in X$ such that $\{x,z\}$ is a linearly independent set. Therefore $%
p_{z}(x)\neq 0$ and the family of seminorms $\{p_{z}:z\in X\}$ generates a
locally convex topology on $X$.\newline

Let $\delta $ be any positive real and let $w$ be a vector in a $2$-normed
space $X.$ A set $B_{u}(w,\delta )=\{x\in X:\left\Vert x-w,u\right\Vert
<\delta \}$ for any $u$ in $X$ is called an \emph{open ball} in $X$, whereas
the set $B_{u}[w,\delta ]=\{x\in X:\left\Vert x-w,u\right\Vert \leq \delta
\} $ for any $u$ in $X$ is called a \emph{closed ball} in $X$ (\cite{Topo}).
A subset $G$ in a $2$-normed space $X$ is said to be an \emph{open set} in $%
X $ if for all $w\in G$, there exists $u\in X$ and $\delta >0$ such that $%
B_{u}(w,\delta )\subseteq G$.

\begin{defn}
(White \cite{white}) A sequence $\{x_{n}\}$ in a $2$-normed space $X$ is
said to be:

\begin{description}
\item[(a)] \emph{convergent} to some $x\in X$ if
\begin{equation*}
\lim_{n\rightarrow \infty }\left\Vert x_{n}-x,y\right\Vert =0
\end{equation*}%
for every $y\in X$;

\item[(b)] a \emph{Cauchy sequence} if
\begin{equation*}
\lim_{m,n\rightarrow \infty }\left\Vert x_{n}-x_{m},z\right\Vert =0
\end{equation*}%
for every $z\in X$.
\end{description}

A $2$-normed space $X$ is said to be $2$\emph{-Banach space} if every Cauchy
sequence in $X$ is convergent in $X$.
\end{defn}

\begin{defn}
(Mal\v{c}eski and Ibrahimi \cite{Malceski}) Let $X$ be a $2$-normed space
and let $A\subseteq X$ be a subset. The \emph{closure} of $A$, denoted by $%
\bar{A}$, is defined as the set%
\begin{equation*}
\left\{ \,x\in X:\exists \ \{x_{n}\}\subseteq A\ \text{ with }%
\lim_{n\rightarrow \infty }x_{n}=x\,\right\} .
\end{equation*}%
The set $A$ is said to be \emph{closed} if $A=\bar{A}$.

Given $a\in A$, the set $A$ is said to be $a$\emph{-bounded} if there exists
a positive real number $\beta >0$ such that $\left\Vert x,a\right\Vert \leq
\beta $ for all $x\in A$. If $A$ is $a$-bounded for each $a\in A,$ then $A$
is said to be a \emph{bounded set}.
\end{defn}

\begin{rem}
\label{newrem}It was demonstrated in \cite{Hatikrishnan}\ that every closed
subset of a $2$-Banach space is complete.
\end{rem}

\begin{defn}
(Hatikrishnan and Ravindran \cite{Hatikrishnan})
Let $E$ be nonempty subset of $2$%
-normed space $X.$  A self mapping $T$ on a $2$%
-normed space $E$ is called:

\begin{description}
\item[(a)] a \emph{contraction} if there exists $\alpha \in \lbrack 0,1)$
such that, for any $x,y,z\in E,$ we have%
\begin{equation}
\left\Vert Tx-Ty,z\right\Vert \leq \alpha \left\Vert x-y,z\right\Vert ;
\label{E1}
\end{equation}

\item[(b)] \emph{nonexpansive} if we take $\alpha =1$ in the above
inequality;

\item[(c)] \emph{sequentially continuous} if $\{Tx_{n}\}$ converges to $Tx$
for any sequence $\{x_{n}\}$ in $E$ such that $\{x_{n}\}$ converges to $x\in
E$.
\end{description}
\end{defn}

\begin{lem}
(Hatikrishnan and Ravindran \cite{Hatikrishnan})\label{seq} Let $%
(X,\left\Vert \cdot ,\cdot \right\Vert )$ be a $2$-normed space. Then every
contraction $T:X\rightarrow X$ is sequentially continuous.
\end{lem}

Hatikrishnan and Ravindran proved the following fixed point result in \cite%
{Hatikrishnan}.

\begin{lem}
(Hatikrishnan and Ravindran \cite{Hatikrishnan})\label{hatlem} Let $%
(X,\left\Vert \cdot ,\cdot \right\Vert )$ be a $2$-Banach space and let $E$
be a nonempty, closed and bounded subset of $X$. If $T:E\rightarrow E$ is a
contraction, then $T$ has a\ unique fixed point. Moreover, for any initial
point $x_{0}\in E$, the Picard iteration $\{x_{n}\}$ associated with $T$
converges to the fixed point of $T$.
\end{lem}

The main aim of this paper is to extend some of the results for enriched
contractions described in \cite{BP} to the more general case of a $2$-normed
spaces.

\section{Enriched contractions in 2-Normed space}

First we introduce the notion of $(b,\theta )$-enriched contraction mapping
in the setting of $2$-normed spaces.

\begin{defn}
Let $E$ be nonempty subset of $2$%
-normed space $(X,\left\Vert \cdot ,\cdot \right\Vert ).$
 A mapping $T:E\rightarrow E$ is said to be a $(b,\theta )$\emph{-enriched
contraction} if there exist $b\in \lbrack 0,\infty )$ and $\theta \in
\lbrack 0,b+1)$ such that, for any $x,y,z\in E,$ we have
\begin{equation}
\left\Vert b(x-y)+Tx-Ty,z\right\Vert \leq \theta \left\Vert x-y,z\right\Vert
.  \label{E2}
\end{equation}
\end{defn}

\begin{exmp}
\label{example}The following examples show that the class of all $(b,\theta
) $-enriched contractions is nonempty.

\begin{enumerate}
\item Any contraction $T$ defined as in (\ref{E1}) with contractive constant
$\alpha \in \lbrack 0,1)$ is a $(0,\alpha )$-enriched contraction.

\item Let $(X,\left\Vert \cdot ,\cdot \right\Vert )$ be a $2$-normed space
and let $w$ be an arbitrary but fixed element of $X$. Define a mapping $%
T:X\rightarrow X$ by%
\begin{equation*}
Tx=w-x.
\end{equation*}%
Then $T$ is not a contraction in the sense of (\ref{E1}). However, let us
show that $T$ is a $(b,\theta )$-enriched contraction. Indeed, in this case,
the enriched contractivity condition (\ref{E2}) is equivalent to
\begin{equation}
|b-1|\left\Vert x-y,z\right\Vert \leq \theta \left\Vert x-y,z\right\Vert
\qquad \text{for all }x,y,z\in X,  \label{E6}
\end{equation}%
where $\theta \in \lbrack 0,b+1)$. The above inequality is valid for all $%
x,y,z\in X$ if one chooses any $b\in \left( 0,1\right) $ and $\theta =1-b$.
Hence, for any $b\in (0,1)$, $T$ is a $(b,1-b)$-enriched contraction. Notice
that $w/2$ is a fixed point of $T$ because $T(w/2)=w/2$.
\end{enumerate}
\end{exmp}

\begin{rem}
We highlight that the Picard iteration $\{x_{n}\}$ associated to the mapping
$T$ defined by (\ref{E6}) in Example \ref{example} does not converge
whatever the initial point $x_{0}$ distinct from $w/2$.
\end{rem}

The previous remark illustrates that the classical Picard iteration is not
useful in order to approximate fixed points of enriched contractions. Hence,
we need to change our point of view. In the rest of the section, we are
going to show that the Krasnoselskij iterative scheme \cite{Borwein} is more
appropriate in this context. In fact, we shall prove a strong convergence
theorem for the class of enriched contractions. Before that, recall the
following concept.

\begin{rem}
\label{rem}Let $A$ be a convex subset of a $2$-normed space $X$ and let $%
T:A\rightarrow A$. Then, for any $\lambda \in (0,1]$, the set of all fixed
points of the averaged mapping $T_{\lambda }:A\rightarrow A$ (given by $%
T_{\lambda }(x)=(1-\lambda )x+\lambda Tx$ for all $x\in X$) coincides with
set of all fixed points of $T$.
\end{rem}

We \ now give a generalization of Lemma \ref{hatlem} for the case of $%
(b,\theta )$-enriched contractions in the setting of $2$-Banach spaces.

\begin{thm}
\label{main}
 Let $E$ be a nonempty, closed, bounded  and convex subset of a $2$%
-Banach space $(X,\left\Vert \cdot ,\cdot \right\Vert )$ and let $%
T:E\rightarrow E$ be a $(b,\theta )$-enriched contraction.   Assume that there exists $x_{0}\in E$ such that
$x_{0}-T_{\lambda}x_{0}\in E,$ then
\begin{enumerate}
\item $Fix(T)=\{x^{\ast }\};$

\item The iterative sequence $\{x_{n}\}_{n=1}^{\infty }$ given by

\begin{equation}
x_{n}=(1-\lambda )x_{n-1}+\lambda Tx_{n-1}\qquad \text{for all }n\in \mathbb{%
N}  \label{EE3}
\end{equation}
\end{enumerate}
converges to $x^{\ast },$  where $\lambda =1/(b+1)$.
\end{thm}

\begin{proof}
Since $\lambda =1/(b+1)$ then $b=(1-\lambda )/\lambda $. Let $T_{\lambda
}:E\rightarrow E$ be given as in (\ref{E7}). Taking into account that $T$ is
a $(b,\theta )$-enriched contraction, for all $x,y,z\in E$,%
\begin{eqnarray*}
\frac{1}{\lambda }\left\Vert T_{\lambda }x-T_{\lambda }y,z\right\Vert &=&%
\frac{1}{\lambda }\left\Vert \,\left[ \,(1-\lambda )x+\lambda Tx\,\right] -%
\left[ \,(1-\lambda )y+\lambda Ty\,\right] ,z\right\Vert \\
&=&\frac{1}{\lambda }\left\Vert (1-\lambda )\left( x-y\right) +\lambda
\left( Tx-Ty\right) ,z\right\Vert \\
&=&\left\Vert \frac{1-\lambda }{\lambda }\left( x-y\right) +\left(
Tx-Ty\right) ,z\right\Vert \\
&=&\left\Vert b\left( x-y\right) +\left( Tx-Ty\right) ,z\right\Vert \leq
\theta \left\Vert x-y,z\right\Vert .
\end{eqnarray*}%
This inequality can be written in an equivalent way as follows:
\begin{equation}
\left\Vert T_{\lambda }x-T_{\lambda }y,z\right\Vert \leq d\left\Vert
x-y,z\right\Vert \qquad \text{for all }x,y,z\in E,  \label{E4}
\end{equation}%
where $d=\theta \lambda $. Since $\theta \in (0,b+1)$, then $d\in (0,1)$.
Notice that the Krasnoselskij iterative process $\{x_{n}\}$ defined by (\ref%
{EE3}) is exactly the Picard iteration associated with $T_{\lambda }$, that
is,
\begin{equation}
x_{n}=T_{\lambda }x_{n-1}=T_{\lambda }^{n}x_{0}\qquad \text{for all }n\in
\mathbb{N}.  \label{E5}
\end{equation}%
Let $m,n>0$ with $m>n.$ Take $p=m-n\in \mathbb{N}$, that is, $m=n+p$. Then,
for any $z\in E$,%
\begin{align*}
\left\Vert x_{n}-x_{m},z\right\Vert & =\left\Vert x_{n}-x_{n+p},z\right\Vert
\\
& =\left\Vert (x_{n}-x_{n+1})+(x_{n+1}-x_{n+2})+x_{n+2}+\dotsc
+(x_{n+p-1}-x_{n+p}),z\right\Vert \\
& \leq \left\Vert x_{n}-x_{n+1},z\right\Vert +\left\Vert
x_{n+1}-x_{n+2},z\right\Vert +\dotsc +\left\Vert
x_{n+p-1}-x_{n+p},z\right\Vert \\
& =\left\Vert T_{\lambda }^{n}x_{0}-T_{\lambda }^{n}x_{1},z\right\Vert
+\left\Vert T_{\lambda }^{n+1}x_{0}-T_{\lambda }^{n+1}x_{1},z\right\Vert
+\dotsc \\
& \ \ \ \ \ \ \ \ \ \ \ \ \ \ \ \ \ \ \ \ \ \ \ \ \ \ \ \dotsc +\left\Vert
T_{\lambda }^{n+p-1}x_{0}-T_{\lambda }^{n+p-1}x_{1},z\right\Vert \\
& \leq d^{n}\left\Vert x_{0}-x_{1},z\right\Vert +d^{n+1}\left\Vert
x_{0}-x_{1},z\right\Vert +\dotsc +d^{n+p-1}\left\Vert
x_{0}-x_{1},z\right\Vert \\
& \leq d^{n}\left\Vert x_{0}-x_{1},z\right\Vert (1+d+d^{2}+\dotsc )=\frac{%
d^{n}}{1-d}\cdot \left\Vert x_{0}-x_{1},z\right\Vert .
\end{align*}%
Thus%
\begin{equation}
\left\Vert x_{n}-x_{m},z\right\Vert \leq \frac{d^{n}}{1-d}\cdot \left\Vert
x_{0}-x_{1},z\right\Vert \qquad \text{for all }z\in E.  \label{567}
\end{equation}%
As $E$ is bounded, there exists $\beta >0$ such
that
\begin{equation*}
\left\Vert \lambda (x_{0}-Tx_{0}),z\right\Vert =\left\Vert x_{0}-T_{\lambda
}x_{0},z\right\Vert \leq \beta
\end{equation*}%
for all $z\in E.$ Hence
\begin{equation*}
\left\Vert x_{0}-x_{1},z\right\Vert \leq \beta ,\ \text{for all }\ z\in E.
\end{equation*}%
From (\ref{567}), we have
\begin{equation*}
\left\Vert x_{n}-x_{m},z\right\Vert \leq \frac{d^{n}\,\beta }{1-d},\ \ \text{%
for all }\ z\in E.
\end{equation*}%
As $m>n$ are arbitrary, letting $n$ tends to $\infty $, we deduce that $%
\{x_{n}\}$ is a  Cauchy sequence in $E$. Since
 there exists $%
x^{\ast }\in E$ such that sequence $\{x_{n}\}$ converges to $x^{\ast }$. By
using (\ref{E4}) and Lemma \ref{seq}, we get $x_{n+1}=T_{\lambda
}x_{n}\rightarrow T_{\lambda }x^{\ast }=x^{\ast }$ as $n\rightarrow \infty $%
. Hence $x^{\ast }$ is a fixed point of $T_{\lambda }$. To prove the
uniqueness of the fixed point of $T_{\lambda }$, assume, on the contrary,
that there exists $y^{\ast }$ in $E$ such that $T_{\lambda }y^{\ast
}=y^{\ast }$and $x^{\ast }\neq y^{\ast }$. It follows from (\ref{E4}) that,
for all $z\in E$,
\begin{equation*}
\left\Vert x^{\ast }-y^{\ast },z\right\Vert =\left\Vert T_{\lambda }x^{\ast
}-T_{\lambda }y^{\ast },z\right\Vert \leq d\left\Vert x^{\ast }-y^{\ast
},z\right\Vert ,
\end{equation*}%
which is a contradiction. Hence the fixed point of $T_{\lambda }$ is unique.
By Remark \ref{rem}, as the sets of fixed points of $T$ and $T_{\lambda }$
coincide, we conclude that $Tx^{\ast }=x^{\ast }$ and that $x^{\ast }$ is
the unique fixed point of $T$.
\end{proof}

\begin{rem}
If $T$ is a $(0,\theta )$-enriched contraction then $T$ is contraction
satisfying, for all $x,y,z\in E,$%
\begin{equation*}
\left\Vert Tx-Ty,z\right\Vert \leq \theta \left\Vert x-y,z\right\Vert .
\end{equation*}%
In this case, the Krasnoselskijj type iterative process (\ref{EE3}) reduces
to the Picard iterative process (because $b=0$ implies $\lambda =1$ in
Theorem \ref{main}). Hence Theorem \ref{main} states that the Picard
iterative process can be applied for contractions in a $2$-metric space.
\end{rem}

If we take $b=0$ in the Theorem \ref{main}, we obtain lemma \ref{hatlem} in
the setting of closed, bounded   and convex subset of $2$-Banach space.

\begin{cor} \label{po}
\cite{Hatikrishnan} Let $(X,\left\Vert \cdot ,\cdot \right\Vert )$ be a $2$%
-Banach space and $E$ be nonempty closed, bounded  and convex subset   of $X$. Let $%
T:E\rightarrow E$ be $(0,\theta )$-enriched contraction, then $T$ has unique
fixed point.
\end{cor}

\section{Local and asymptotic versions of the enriched contractive principle}

In concrete applications, the local variant of the enriched contraction
(which involves an open ball $B$ in a Banach space $X$ and a nonself
enriched contractive map of $B$ into $X$ with an essential property that it
does not displace the center of the ball too far) is important (see \cite{BP}%
). The analog of this result in the case of $(b,\theta )$-enriched
contractions in $2$-normed spaces is given by the following result.

\begin{cor}
\label{local} Let $(X,\left\Vert \cdot ,\cdot \right\Vert )$ be a $2$-Banach
space, $x_{0},u\in X$, $r>0$, $B=B_{u}(x_{0},r)$ and $T:B\rightarrow X$ a $%
(b,\theta )$-enriched contraction. If
\begin{equation}
\left\Vert x_{0}-Tx_{0},u\right\Vert <(b+1-\theta )r,  \label{Ball}
\end{equation}%
then $T$ has unique fixed point provided that $B$ is bounded and $T_{\lambda
}(B)\subseteq B$, where $\lambda =\frac{1}{b+1}.$
\end{cor}

\begin{proof}
Let $\varepsilon \in \left( 0,r\right) $ be such that
\begin{equation*}
\left\Vert x_{0}-Tx_{0},u\right\Vert <(b+1-\theta )\varepsilon <(b+1-\theta
)r.
\end{equation*}%
In a similar manner to the proof of Theorem \ref{main}, one obtains that $%
T_{\lambda }$ is a $(0,d)$-enriched contraction on $B$, that is, we have
\begin{equation}
\left\Vert T_{\lambda }x-T_{\lambda }y,z\right\Vert \leq d\left\Vert
x-y,z\right\Vert \qquad \text{for all }x,y,z\in B,  \label{nk}
\end{equation}%
where%
\begin{equation*}
d=\theta \lambda =\frac{\theta }{b+1}.
\end{equation*}%
\newline
From the inequality (\ref{Ball}), one obtains that
\begin{equation}
\left\Vert x_{0}-T_{\lambda }x_{0},u\right\Vert <\left( 1-\frac{\theta }{b+1}%
\right) \varepsilon =(1-d)\varepsilon \qquad \text{for all }z\in X.
\label{ball2}
\end{equation}%
We now prove that the closed ball
\begin{equation*}
A=B_{u}[x_{0},\varepsilon ]=\{\,x\in X:\left\Vert x_{0}-x,u\right\Vert \leq
\varepsilon \,\}\subseteq B
\end{equation*}%
is invariant under $T_{\lambda }$. Indeed, for any $x\in A$, using
triangular inequality, (\ref{nk}) and (\ref{ball2}), we have
\begin{align*}
\left\Vert x_{0}-T_{\lambda }x,u\right\Vert & =\left\Vert x_{0}-T_{\lambda
}x_{0}+T_{\lambda }x_{0}-T_{\lambda }x,u\right\Vert \\
& \leq \left\Vert x_{0}-T_{\lambda }x_{0},u\right\Vert +\left\Vert
T_{\lambda }x_{0}-T_{\lambda }x,u\right\Vert \\
& <(1-d)\varepsilon +d\left\Vert x_{0}-x,z\right\Vert \leq (1-d)\varepsilon
+d\varepsilon =\varepsilon ,
\end{align*}%
that is, $T_{\lambda }x\in A,$ for any $x\in A.$ The conclusion follows by
Theorem \ref{main}.
\end{proof}

We now present the local version in the setting of a nonempty, closed,
convex and bounded subset of a $2$-Banach space.

\begin{thm}
Let $E$ be a nonempty, closed, convex and bounded subset of a $2$-Banach
space $(X,\left\Vert \cdot ,\cdot \right\Vert )$ and $T:E\rightarrow X$ be a
$(b,\theta )$-enriched contraction. Then $T$ has a unique fixed point
provided that $b>0$.
\end{thm}

\begin{proof}
Since $b>0$, then $\lambda =\frac{1}{b+1}\in \left( 0,1\right) $. As $E$ is
convex, $T_{\lambda }x\in E$ for all $x\in E$. As $T$ is a $(b,\theta )$%
-enriched contraction, one obtains from condition (\ref{E2}) that $%
T_{\lambda }$ is a $(0,d)$-enriched contraction on $E$, where
\begin{equation*}
d=\theta \lambda =\frac{\theta }{b+1}.
\end{equation*}%
Since $\theta \in (0,b+1)$, then $d\in (0,1)$. Let $x_{0}\in E$. Then we can
construct a sequence $\{x_{n}\}_{n\in \mathbb{N}}$ in a similar manner as in
the proof of Theorem \ref{main}, and one obtains
\begin{equation}
\left\Vert x_{n}-x_{m},z\right\Vert \leq \frac{d^{n}}{1-d}\cdot \left\Vert
x_{0}-x_{1},z\right\Vert \qquad \text{for all }z\in E.  \label{AAA}
\end{equation}%
Hence, it can be proved that $\{x_{n}\}$ is a Cauchy sequence. Following
similar arguments to those given in the proof of Theorem \ref{main}, the
result follows.
\end{proof}

The following example shows that there exist mappings $T$ which is not an
enriched contraction but its certain iterate is an enriched contraction.

\begin{exmp}
Let $(X,\left\Vert \cdot ,\cdot \right\Vert )$ be a $2$-normed space.
Suppose that $X$ can be decomposed into $X=A\cup B$, where $A\cap
B=\emptyset $. Also assume that there is $u\in B$ such that $-\frac{u}{3}\in
B$. Let define $T:X\rightarrow X$ as%
\begin{equation*}
Tx=%
\begin{cases}
u & \text{if $x\in A$}, \\
-\frac{u}{3} & \text{if $x\in B.$}%
\end{cases}%
\end{equation*}%
If $T$ was an enriched contraction then, by putting $\lambda =\frac{1}{b+1}$
in (\ref{E2}), one would obtain that $T_{\lambda }$ is a contraction on $X$,
where\newline
\begin{equation*}
T_{\lambda }x=%
\begin{cases}
(1-\lambda )x+\lambda u, & \text{if $x\in A$}, \\
(1-\lambda )x-\frac{\lambda u}{3}, & \text{if $x\in B.$}%
\end{cases}%
\end{equation*}%
Clearly, $T_{\lambda }$ is not contraction because it is not sequentially
continuous, so $T$ is not an enriched contraction. However, it is easy to
check that $T^{2}$ is a $(1,1)$-enriched contraction. In fact, notice that $%
T^{2}x=-\frac{u}{3}$ for all $x\in X$.
\end{exmp}

Our next theorem is the more general case of $(b,\theta )$-enriched
contractions in the setting of $2$-Banach spaces.

\begin{thm}
\label{itet}
Let $E$ be nonempty, closed, bounded  and convex subset   subspace of a $2$-Banach space  $(X,\left\Vert \cdot ,\cdot \right\Vert )$
and  $T:E\rightarrow E$ be a mapping with the property there exists
a positive integer $N$ such that $T^{N}$ is a $(b,\theta )$-enriched
contraction.  Assume that there exists $x_{0}\in E$ such that $x_{0}-T_{\lambda}^{N}x_{0}\in E.$
 Then

\begin{enumerate}
\item $Fix(T)=\{x^{\ast }\}$;

\item the iterative method $\{x_{n}\}_{n=1}^{\infty }$ given by
\begin{equation}
x_{n}=(1-\lambda )x_{n-1}+\lambda T^{N}x_{n-1}\qquad \text{for all }n\in
\mathbb{N}  \label{E3}
\end{equation}
\end{enumerate}
converges to $x^{\ast },$ where $\lambda =1/(b+1)$.
\end{thm}

\begin{proof}
We apply Theorem \ref{main} to the mapping $T=T^{N}$ and obtain that $Fix(T^{N})=\{x^{\ast }\}$. We also have
\begin{equation*}
T^{N}(T(x^{\ast }))=T^{N+1}(x^{\ast })=T(T^{N}x^{\ast })=Tx^{\ast },
\end{equation*}%
which shows that $Tx^{\ast }$ is another fixed point of $T^{N}$. But as $%
T^{N}$ has a unique fixed point, which is $x^{\ast }$, then $T(x^{\ast
})=x^{\ast }$. The remaining part of the proof follows from Theorem \ref%
{main}.
\end{proof}

\section{Conclusions}

\begin{enumerate}
\item In this paper, we have introduced and studied the class of $(b,\theta )
$-enriched contractions in the setting of $2$-normed spaces, which includes
the family of contractive mappings, some nonexpansive and Lipschitzian
mappings in $2$-normed spaces as particular cases.

\item We have showed that any $(b,\theta )$-enriched contraction (Theorem %
\ref{main},  in a $2$-normed space has a
unique fixed point that can be approximated by means of a Kransnoselskij
type iterative process.

\item We have also proved in Corollary \ref{po} a generalization of Lemma
3.10 given in \cite{Hatikrishnan} in the setting of $2$-Banach spaces.

\item Finally, we have obtained a local fixed point result (Corollary \ref%
{local}) for $(b,\theta )$-enriched contractions in the setting of $2$%
-Banach spaces. This result significantly extends the corresponding ones in
\cite{BP} from Banach spaces to $2$-Banach spaces as well as an asymptotic
version (Theorem \ref{itet}) of the $(b,\theta )$-enriched contraction
mapping in $2$-normed spaces.
\end{enumerate}

\end{document}